\documentclass[a4paper,10pt]{article}
\usepackage[utf8x]{inputenc}
\usepackage{amsmath,amssymb, amsthm}
\usepackage{enumerate}
\usepackage{verbatim}

\newcommand{\Z}{\mathbb{Z}}

\newcommand{\R}{\mathbb{R}}
\newcommand{\C}{\mathbb{C}}

\newcommand{\OreGen}{R[x;\sigma,\delta]}

\newcommand{\DiffPol}{R[x;\identity_R,\delta]}

\DeclareMathOperator{\End}{End}
\DeclareMathOperator{\identity}{id}

\DeclareMathOperator{\CHAR}{char}

\makeatletter
\def\imod#1{\allowbreak\mkern10mu({\operator@font mod}\,\,#1)}
\makeatother

\theoremstyle{plain}
\newtheorem{theorem}{Theorem}[section]

\newtheorem{prop}[theorem]{Proposition}
\newtheorem{corollary}[theorem]{Corollary}

\theoremstyle{definition}
\newtheorem{definition}[theorem]{Definition}

\newtheorem{exmp}[theorem]{Example}

\newtheorem{remark}[theorem]{Remark}

\title{Burchnall-Chaundy theory for Ore extensions}
\author{Johan Richter % 
\footnote{Centre for Mathematical Sciences, Lund University, Box 118, 22100 Lund, Sweden, E-mail: johanr@maths.lth.se }
}

\begin{document}

\maketitle

\begin{abstract}
We begin by reviewing a classical result on the algebraic dependence of commuting elements in Weyl algebras. We proceed by describing generalizations of this result to various classes of Ore extensions, both results that have
already been published and a new result. 



\end{abstract}

\section{Introduction}

Let $R$ be a commutative ring and $S$ an $R$-algebra. Let $a,b$ be two commuting elements of $S$. We are interested in the question whether they are algebraically dependent over $R$. I.e., does there exist a non-zero polynomial 
$f(s,t)\in R[s,t]$ such that $f(a,b)=0$? Furthermore, can we find a proper subring $F$ of $R$ such that $a,b$ are algebraically dependent over $F$? 

In this article $S$ will typically be an Ore extension of $R$. We start by introducing the notations and conventions we will use in this article and define what an Ore extension is. After that we review without giving proofs results
obtained by other authors for the case that $S$ is a differential operator ring (a special case of Ore extensions). We then proceed to describe results obtained by the present author and his collaborators and we finish by describing 
a strengthening of these results we recently obtained.

\subsection{Notation and conventions}

$\R$ will denote the field of real numbers, $\C$ the field of complex numbers. $\Z$ will denote the integers. 

If $R$ is a ring then $R[x_1,x_2, \ldots x_n]$ denotes the ring of polynomials over $R$ in central indeterminates $x_1,x_2,\ldots,x_n$. 

By a ring we will always mean an associative and unital ring. All morphisms between rings are assumed to map the multiplicative identity element to the multiplicative identity element.  

By an ideal we shall mean a two-sided ideal.

If $R$ is a ring we can regard it as a module (indeed algebra) over $\Z$ by defining $0r =0$, $nr= \sum_{i=1}^n r$ if $n> 0$ and $nr = -(-n)r$ if $n$ is a negative integer. If there is a positive integer $n$ such that $n1_{R}=0$,
we call the smallest such positive integer the \emph{characteristic} of $R$ and denote it $\CHAR(R)$. If no such integer exists we set $\CHAR(R)=0$.  

Let $R$ be a commutative ring and $S$ an $R$-algebra. Two commuting elements, $p,q \in S$, are said to be \emph{algebraically dependent} (over $R$) if there is a non-zero polynomial, $f(s,t) \in R[s,t]$, such that $f(p,q)=0$, in which 
case $f$ is called an annihilating polynomial. 

If $S$ is a ring and $a$ is an element in $S$, the \emph{centralizer} of $a$, denoted $C_S(a)$, is the set of all elements in $S$ that commute with $a$. 

This article studies a class of rings called Ore extensions. For general references on Ore extensions, see e.g. \cite{CGoodearl2004,CRowenRing1}. We shall briefly recall the definition. 
If $R$ is a ring and $\sigma$ is an endomorphism of $R$, then an additive map $\delta: R \to R$ is said to be a $\sigma$-\emph{derivation} if 
\begin{displaymath}
 \delta(ab) = \sigma(a)\delta(b)+\delta(a)b
\end{displaymath}
holds for all $a,b \in R$.

\begin{definition}
Let $R$ be a ring, $\sigma$ an endomorphism of $R$ and $\delta$ a $\sigma$-derivation. The \emph{Ore extension} $\OreGen$ is defined as the ring generated by $R$ and an element $x \notin R$ such that $1,x, x^2, \ldots$ form a
 basis for $\OreGen$ as a left $R$-module and all $r \in R$ satisfy
\begin{equation}\label{CMultiplicationRule}
	x r = \sigma(r)x + \delta(r).
\end{equation}
\end{definition}
Such a ring always exists and is unique up to isomorphism (see \cite{CGoodearl2004}). 
From $\delta(1 \cdot 1) = \sigma(1) \cdot 1 +\delta(1) \cdot 1$ we get that $\delta(1)=0$, and since $\sigma(1)=1$ we see that $1_R$ will be a multiplicative identity for $\OreGen$ as well. 

Any element $r$ of $R$ such that $\sigma(r)=r$ and $\delta(r)=0$ will be called a \emph{constant}. In any ring with an endomorphism $\sigma$ and a $\sigma$-derivation $\delta$ the constants form a subring. 

If $\sigma= \identity_R$, then a $\sigma$-derivation is simply called a derivation and $\DiffPol$ is called a \emph{differential operator ring}.

An arbitrary non-zero element $P \in \OreGen$ can be written uniquely
as $P = \sum_{i=0}^n a_i x^i$ for some $n \in \Z_{\geq 0}$, with $a_i \in R$ for $i \in \{0,1,\ldots, n\}$ and $a_n \neq 0$. The \emph{degree} of $P$ will be defined as $\deg(P):=n$. We set $\deg(0) := -\infty$.

\section{Burchnall-Chaundy theory for differential operator rings}\label{sec_BC}

We shall begin by describing some results on the algebraic dependence of commuting elements in differential operator rings. As the title of this subsection suggests, this sort of question has its origin in a series of papers by the 
British mathematicians Joseph Burchnall and Theodore Chaundy \cite{BurchnallChaundy1,BurchnallChaundy2,BurchnallChaundy3}.

\begin{prop}
 Let $R$ be a ring and $\delta: R \to R$ a derivation. Let $C$ be the set of constants of $\delta$. Then 
\begin{enumerate}[{\rm (i)}]
 \item $1 \in C$;
 \item $C$ is a subring of $R$, called the ring of constants;
 \item for any $c \in C$ and $r \in R$ we have
\begin{align*}
 \delta(cr) &= c \delta(r), \\
 \delta(rc) &= \delta(r) c.  
\end{align*}

\end{enumerate}

\end{prop}

\begin{proof}
We skip the simple calculational proof. 
\end{proof}

\begin{comment}
\begin{proof}
\begin{enumerate}[{\rm (i)}]
 \item \begin{displaymath} 
        \delta(1) = \delta(1 \cdot 1) = 1\delta(1) +\delta(1)1 = \delta(1)+\delta(1) \Rightarrow \delta(1)=0.
       \end{displaymath}
\item If $a,b \in C$ it follows from additivity that $a+b$ is a constant as well. To prove that $ab \in C$ we use the Leibniz rule as follows:
\begin{displaymath}
 \delta(ab) = a\delta(b)+\delta(a)b=0+0=0.
\end{displaymath}
Since $1$ and $0$ are constants it is clear that $C$ is a subring. 

\item Let $c$ be a constant and $r$ any element of $R$. A short calculation gives
\begin{displaymath}
 \delta(cr) = c\delta(r)+\delta(c)r = c\delta(r). 
\end{displaymath}
The other claim is proved analogously.
\end{enumerate}
\end{proof}
\end{comment}

As expected any derivation satisfies a version of the quotient rule. 

\begin{prop}
 Let $R$ be a ring with a derivation, $\delta$, and let $a$ be any invertible element of $R$. Then
\begin{displaymath}
 \delta(a^{-1}) = -a^{-1}\delta(a)a^{-1}.
\end{displaymath}
\end{prop}

\begin{proof}
\begin{displaymath}
 0 = \delta(1) = \delta(a^{-1}a) = a^{-1}\delta(a)+\delta(a^{-1})a \Rightarrow \delta(a^{-1}) =  -a^{-1}\delta(a)a^{-1}.
\end{displaymath}
 \end{proof}

\begin{corollary}
 Let $R$ be a ring with a derivation $\delta$ and $C$ its ring of constants. If $a$ is an invertible element that lies in $C$, then so does $a^{-1}$. If $R$ is a field, then $C$ is a subfield of $R$. 
\end{corollary}

\begin{exmp}\label{C-infty}
As the ring $R$ we can take $C^{\infty}(\R,\C)$, the ring of all infinitely many times differentiable complex-valued functions on the real line. For $\delta$ we can take the usual derivative. 
The ring of constants in this case will consist of 
the constant functions.
\end{exmp}

With $R$ and $\delta$ as in Example \ref{C-infty} we can form the differential operator ring $\DiffPol$. We will show that the name ``differential operator ring'' is apt by constructing a ring of concrete differential operators
that is isomorphic to $\DiffPol$. 

The ring $R=C^{\infty}(\R,\C)$ can be seen as a vector space over $\C$, with operations defined pointwise. So we can consider the ring $\End_{\C}(R)$ of all linear endomorphisms of $R$. (Note that 
the endomorphisms are not required to be multiplicative.)
 $\End_{\C}(R)$ is in turn an 
algebra over $R$.
One of the operators in $\End_{\C}(R)$ is the derivation operator, which we denote by $D$. Furthermore, for any $f \in R$ there is the multiplication operator $M_f$ that maps any function $g \in R$ to $fg$. 
The operator $D$ and all the $M_f$ together generate a subalgebra of $\End_{\C}(R)$, which we denote by $T$.

It is clear that the set of all $M_f$, for $f \in R$, is a subalgebra of $T$, isomorphic to $R$. Thus we abuse notation and identify $M_f$ with $f$. By doing this we can write any element of $T$ as a finite sum, 
$\sum_{i=0}^n a_i D^i$,  
where each $a_i$ is a function in $C^{\infty}(\R,\C)$. Furthermore such a decomposition is unique, or in other words: the powers of $D$ form a basis for $T$ as a free module over $R$.

We now compute the commutator of $D$ and $f$ for any $f \in R$. We temporarily revert to writing $M_f$ for the element in $T$ to make our calculations easier to understand. Let $g$ be an arbitrary function in $R$. We find that 
\begin{eqnarray*}
 (D M_f -M_f D)(g) = D M_f(g) -M_f D(g) = D(fg) -M_f(g') = \\
 f'g+fg'-fg'= f'g= M_{\delta(f)}(g). 
\end{eqnarray*}
Hence
\begin{displaymath}
 D M_f - M_f D = M_{\delta(f)}.
\end{displaymath}
Relapsing into our abuse of notation we write this as $ Df-fD= \delta(f)$ or equivalently as $Df = fD +\delta(f)$. 

Denote the identity function on the real line by $y$. Then $Dy -yD= 1$, a relation known as the Heisenberg relation. The elements $y$ and $D$ together generate a subalgebra of $T$ known as the Weyl algebra or 
the Heisenberg algebra, which is of 
interest in quantum mechanics, among other areas.

Any element, $P$, of $T$ can be written as $P = \sum_{i=0}^n p_i D^i$, for some non-negative integer $n$ and some  $p_i\in C^{\infty}(\R,\C)$.  Conversely every such sum is an element of $T$. Thus $T$ is isomorphic to $\DiffPol$ 
with $R$ and $\delta$ defined as in Example \ref{C-infty}.

In a series of papers in the 1920s and 30s \cite{BurchnallChaundy1,BurchnallChaundy2, BurchnallChaundy3}, Burchnall and Chaundy studied the properties of commuting pairs of ordinary differential operators. In our terminology they may be said to study the properties of pairs of commuting 
elements of $T$. (They do not specify what function space their differential operators are supposed to act on.) The following theorem is essentially found in their papers. 
\begin{theorem}\label{thm_BC}
 Let $P=\sum_{i=0}^n p_i D^i$ and $Q= \sum_{j=0}^{m} q_j D^j$ be two commuting elements of $T$ with constant leading coefficients. Then there is a non-zero polynomial $f(s,t)$ in two commuting variables over $\C$ such that 
$f(P,Q) =0$. Note that the fact that $P$ and $Q$ commute 
guarantees that $f(P,Q)$ is well-defined.
\end{theorem}

The result of Burchnall and Chaundy was rediscovered independently during the 1970s by researchers in the area of PDEs. It turns out that several important PDEs are equivalent to the condition that a pair of differential 
operators commute. These differential equations are completely integrable as a result, which roughly means that they possess an infinite number of conservation
laws.

Burchnall's and Chaundy's work rely on analytical facts, such as the existence theorem for solutions of linear ordinary differential equations. However, it is possible to give algebraic proofs for the existence of 
the annihilating polynomial. 
This was done later by authors such as Amitsur \cite{Amitsur} and Goodearl \cite{GoodearlPseudo, GoodearlCarlson}. Once one casts Burchnall's and Chaundy's results in an algebraic form one can also generalize them
to a broader class of rings.

More specifically, one can prove Burchnall's and Chaundy's result for certain differential operator rings.
 
Amitsur \cite[Theorem 1]{Amitsur} (following work of Flanders \cite{Flanders}) studied the case when $R$ is a field of characteristic zero and $\delta$ is an arbitrary derivation on $R$. He obtained the following theorem. 
\begin{theorem}\label{thm_DiffField}
 Let $k$ be a field of characteristic zero with a derivation $\delta$. Let $F$ denote the subfield of constants. Form the differential operator ring $S=k[x;\identity,\delta]$, and let $P$ be an element of $S$ of 
degree $n$. Denote by by $F[P]$ the ring of polynomials in $P$ with constant coefficients, $F[P]= \{ \sum_{j=0}^m b_j P^j \ | \ b_j \in F \ \}$. 
Then $C_S(P)$ is a commutative subring of
$S$ and a free $F[P]$-module of rank at most $n$.  
\end{theorem}
 
The next corollary can be found in \cite[Corollary 2]{Amitsur}. 

\begin{corollary}\label{cor_Amitsur}
 Let $P$ and $Q$ be two commuting elements of $k[x;\identity, \delta]$, where $k$ is a field of characteristic zero. Then there is a nonzero polynomial $f(s,t)$, with coefficients in $F$, such that $f(P,Q)=0$. 
\end{corollary}

\begin{proof}
 Let $P$ have degree $n$. Since $Q$ belongs to $C_S(P)$ we know that $1,Q, \ldots , Q^n$ are linearly dependent over $F[P]$ by Theorem \ref{thm_DiffField}. But this tells us that there are elements 
$\phi_0(P), \phi_1(P), \ldots \phi_n(P),$ in $F[P]$, of which not all are zero,
 such that 
\begin{displaymath}
 \phi_0(P)+\phi_1(P)Q+\ldots + \phi_n(P)Q^n = 0.
\end{displaymath}
Setting $f(s,t)=\sum_{i=0}^n \phi_i(s)t^i$ the corollary is proved. 
\end{proof}

\begin{remark}
 Note that $F$, the field of constants, equals the center of $\DiffPol$.
\end{remark}

In \cite{GoodearlPseudo} Goodearl has extended the results of Amitsur to a more general setting. The following theorem is contained in \cite[Theorem 1.2]{GoodearlPseudo}.
\begin{theorem}\label{GoodearlThm}
 Let $R$ be a semiprime commutative ring with derivation $\delta$ and assume that its ring of constants is a field, $F$. If $P$ is an operator in $\DiffPol$ of positive degree $n$, where $n$ is invertible 
in $F$, and has an invertible leading 
coefficient, then $C_S(P)$ is a free $F[P]$-module of rank at most $n$.
\end{theorem}
 We recall that a commutative ring is semiprime if and only if it has no nonzero nilpotent elements.

Goodearl notes that if $R$ is a semiprime ring of positive characteristic such that the ring of constants is a field, then $R$ must be a field. In this case he proves the following theorem \cite[Theorem 1.11]{GoodearlPseudo}.

\begin{theorem}\label{GoodearlThm2}
 Let $R$ be a field, with a derivation $\delta$, and let $F$ be its subfield of constants. If $P$ is an element of $S=\DiffPol$ of positive degree $n$ and with invertible leading coefficient, then $C_S(P)$ is a free 
$F[P]$-module of rank at most $n^2$. 
\end{theorem}

 As before we get the following corollary (of both Theorem \ref{GoodearlThm} and Theorem \ref{GoodearlThm2}), which is found in  \cite[Theorem 1.13]{GoodearlPseudo}.
\begin{corollary}
 Let $P$ and $Q$ be commuting elements of $\DiffPol$, where $R$ is a semiprime commutative ring, with a derivation $\delta$ such that the subring of constants is a field.  Suppose that the leading coefficient of $P$ is invertible. 
Then there 
exists a non-zero polynomial 
$f(s,t) \in F[s,t]$ such that $f(P,Q)=0$.
\end{corollary}

Note that Amitsur's work does not quite generalize Burchnall's and Chaundy's results since $C^\infty(\R,\C)$ is not a field. Theorem~\ref{GoodearlThm} does however imply their results since $C^\infty(\R,\C)$ is 
certainly commutative, 
does not have any nonzero nilpotent elements and its ring of constants is a field (isomorphic to $\C$). The only point to notice is that Theorem \ref{GoodearlThm} requires $P$ to have positive degree. If $P$ is an element of degree
zero and with constant leading coefficient however, it is itself a constant. Then $f(s,t) = s-P$ will be an annihilating polynomial for $P$ and any $Q$. 

An earlier paper by Carlson and Goodearl, \cite{GoodearlCarlson}, contains results similar to Theorems \ref{GoodearlThm} and \ref{GoodearlThm2}, in a different setting. Part of their Theorem 1 can be formulated as follows.  
\begin{theorem}
 Let $R$ be a commutative ring, with a derivation $\delta$ such that the ring of constants is a field, $F$, of characteristic zero. Assume that, for all $a \in R$, 
if the set $\{ r \in R \ | \ \delta(r) = ar \ \}$ contains a nonzero element, then it contains an invertible element. Let $P$ be an element of $\DiffPol$ of positive degree $n$ with invertible leading coefficient. Then $C_S(P)$ 
is a free $F[P]$-module of rank at most $n$. As before, this implies that if $Q$ commutes with $P$, there exists a nonzero polynomial $f(s,t) \in F[s,t]$ such that $f(P,Q)=0$.  
\end{theorem}

Note that the ring $R$ in Example \ref{C-infty} satisfies the conditions of the theorem. 

\section{Burchnall-Chaundy theory for Ore extensions}

Let $k$ be a field and $q$ a nonzero element of that field, not a root of unity. Set $R = k[y]$, a polynomial ring in one variable over $k$. There is an endomorphism $\sigma$ of $R$ such that $\sigma(y)=qy$ 
and the restriction of $\sigma$ to $k$ is the identity. For this $\sigma$ there exists a unique $\sigma$-derivation $\delta$ such that $\delta(y)=1$ and $\delta(\alpha)=0$ for all $\alpha \in k$. The Ore extension 
$\OreGen$ for this choice of $R, \sigma$ and $\delta$ is known as the (first) $q$-Weyl algebra. (An alternative name is the $q$-Heisenberg algebra.)

Silvestrov and collaborators \cite{Silvestrov, HSbook, LarssonSilv} have extended the result of Burchnall and Chaundy to the $q$-Weyl algebra. The cited references contain two different proofs for the fact that any pair of
commuting elements of $\OreGen$ are algebraically dependent over $k$. In \cite{Silvestrov} an algorithm is given to compute an annihilating polynomial explicitly. 

The algorithm is a variation of one presented by Burchnall and Chaundy in their original papers and consists of forming a certain determinant that when evaluated gives the annihilating polynomial. 

Mazorchuk \cite{Mazorchuk} has presented an alternative approach to showing the algebraic dependence of commuting elements in $q$-Weyl algebras. He proves the following theorem. 

\begin{theorem}\label{thm_Maz}
Let $k$ be a field and $q$ an element of $k$. Set $R=k[y]$ and suppose that $\sum_{i=0}^N q^i \neq 0$ for any natural number $N$. Let $P$ be an element of $S=\OreGen$ of degree at least $1$. Then $C_S(P)$ is a 
free $k[P]$-module of finite rank. 
\end{theorem}

If $P$ is as in the theorem and $Q$ is any element of $\OreGen$ that commutes with $P$, then there is an annihilating polynomial $f(s,t)$ with coefficients 
in $k$. This is proven in the same way as Corollary \ref{cor_Amitsur}. The methods used to obtain Theorem \ref{thm_Maz} have been generalized by Hellstr\"{o}m and Silvestrov in \cite{ergodipotent}. 

In \cite[Theorem 3]{RichterSilvestrov} Silvestrov and the present author extend the algorithmic method of \cite{Silvestrov} to more general Ore extensions. 

\begin{theorem}
 Let $R$ be an integral domain with an injective endomorphism $\sigma$ and a $\sigma$-derivation $\delta$. Let $a,b$ be two commuting elements of $\OreGen$. Then there exists a nonzero polynomial  $f(s,t) \in R[s,t]$ such that
$f(a,b)=0$. 
\end{theorem}

Note that if we apply this theorem to the $q$-Weyl algebra with $R=k[y]$ we get a weaker result than the one stated above. We would like to be able to conclude that if $a,b$ are commuting elements of $k[y][x;\sigma,\delta]$ then 
there is a polynomial $f(s,t)$ in $k[s,t]$ such that $f(a,b)=0$. 

Under certain assumptions on $\sigma$ we have been able to prove this and we now proceed to describe how. We begin with a general theorem that we use as a lemma. 

\begin{theorem}\label{thm_Centralizer}
 Let $R$ be an integral domain, $\sigma$ an injective endomorphism of $R$ and $\delta$ a $\sigma$-derivation on $R$. Suppose that the ring of constants, $F$, is a field. Let $a$ be an element of $S=\OreGen$ of degree $n$
and assume that if $b$ and $c$ are two elements in $C_S(a)$ such that $\deg(b)=\deg(c)=m$, then $b_m=\alpha c_m$, where $b_m$ and $c_m$ are the leading coefficients of $b$ and $c$ respectively, and $\alpha$ is some constant.

Then $C_S(a)$ is a free $F[a]$-module of rank at most $n$.     
\end{theorem}

The proof we give is the same as used in \cite{GoodearlPseudo} to prove Theorem \ref{GoodearlThm}.

\begin{proof} 

Denote by $M$ the subset of elements of $\{0, 1, \ldots, n-1 \}$ such that an integer $0 \leq i < n$ is in $M$ if and only if $C_S(a)$ contains an element of degree equivalent to $i$ modulo $n$. For $i \in M$ let $p_i$ be an element in $C_S(a)$ such that 
$\deg(p_i) \equiv i   \imod n$ and $p_i$ has minimal degree for this property. Take $p_0=1$. 

We will show that $\{ p_i | i \in M \}$ is a basis for $C_S(a)$ as a $F[a]$-module. 

Since $R$ is an integral domain and $\sigma$ is injective, the degree of a product of two elements in $\OreGen$ is the sum of the degrees of the two elements. 

We start by showing that the $p_i$ are linearly independent over $F[a]$. Suppose $\sum_{i \in M} f_i p_i =0$ for some $f_i \in F[a]$. If $f_i \neq 0$ then $\deg(f_i)$ is divisible by $n$, in which case 
\begin{align}
 \deg(f_i p_i)= \deg(f_i)+\deg(p_i) \equiv \deg(p_i) \equiv i  \imod n. 
\end{align}
If $\sum_{i \in M} f_i p_i=0$ but not all $f_i$ are zero, we must have two nonzero terms, $f_i p_i$ and $f_j p_j$, that have the same degree despite $i,j \in M$ being distinct. But this is impossible since $ i \not\equiv j 
\imod n$. 

We now proceed to show that the $p_i$ span $C_S(a)$. Let $W$ denote the submodule they do span. We use induction on the degree to show that all elements of $C_S(a)$ belong to $W$. If $e$ is an element of degree $0$ in 
$C_S(a)$ we find by
the hypothesis on $a$ applied to $e$ and $p_0=1$ that $e = \alpha $ for some $\alpha \in F$. Thus $e \in W$. 

Now assume that $W$ contains all elements in $C_S(a)$ of degree less than $j$. Let $e$ be an element in $C_S(a)$ of degree $j$. There is some $i$ in $M$ such that $j \equiv i \imod n$. Let $m$ be the degree of $p_i$. By the 
choice of $p_i$ we now that $m \equiv j \imod n$ and $m \leq j$. Thus $j= m+qn$ for some non-negative integer $q$. The element $a^q p_i$ lies in $W$ and has degree $j$. By hypothesis, the leading coefficient of $e$ equals 
the leading coefficient of $a^q p_i$ times some constant $\alpha$. The element $e-\alpha a^q p_i$  then lies in $C_S(a)$ and has degree less than $j$. By the induction hypothesis it also lies in $W$, and hence so does $e$. 
\end{proof}

We aim to use Theorem \ref{thm_Centralizer} when $R=k[y]$. To that end we have obtained the following proposition.

\begin{prop}\label{conj_main}
 Let $k$ be a field and set $R=k[y]$. Let $\sigma$ be an endomorphism of $R$ such that $\sigma(\alpha)=\alpha$ for all $\alpha \in k$ and $\sigma(y)=p(y)$, where $p(y)$ is a polynomial of degree (in $y$) greater than $1$. 
Let $\delta$ be a $\sigma$-derivation such that $\delta(\alpha)=0$ for all $\alpha \in k$. Form the Ore extension $S=\OreGen$. We note that its ring of constants is $k$. 
Let $a \neq 0$ be an element of $\OreGen$. Assume that $b,c$ are elements of $S$ such that $\deg(b)=\deg(c)=m$ (here the degree is taken with respect to $x$) and $b,c$ both belong to $C_S(a)$. Then $b_m = \alpha c_m$, where
$b_m,c_m$ are the leading coefficients of $b$ and $c$ respectively, and $\alpha$ is some constant. 
\end{prop}

The author wishes to thank Fredrik Ekstr\"om for contributing a crucial idea to the following proof. 

\begin{proof}
Let  $a_n$ be the leading coefficient of $a$. By comparing the leading coefficient of $ab$ and $ba$ we see that
\begin{align}\label{eq_1}
 a_n \sigma^n(b_m) = b_m \sigma^m(a_n).
\end{align}
Similarly 
\begin{align} \label{eq_2}
 a_n \sigma^n(c_m) = c_m \sigma^m(a_n).
\end{align}

By dividing Equation \ref{eq_1} by Equation \ref{eq_2} we see that 
\begin{align}\label{eq_3}
\frac{ \sigma^n(b_m) }{ \sigma^n(c_m)} = \frac{b_m }{ c_m}. 
\end{align}
We can perform such a division by passing to the quotient field of $k[y]$.
 
It thus suffices to prove that if $f,g,p$ are polynomials in $k[y]$, with $\deg(p)>1$, and 
\begin{align}\label{eq_4}
 f(y)g(p(y))=f(p(y))g(y),
\end{align}
then $f(y)=\alpha g(y)$ for some $\alpha \in k$. 

So suppose that such $f,g$ and $p$ are given. We will also assume that $k$ is algebraically closed, which can be done without loss of generality. 
 If $f$ and $g$ have a common factor $h$ we write $f(y)= h(y) \hat{f}(y)$ and similarly for $g$. We find that
\begin{align}
 \hat{f}(y)h(y) h(p(y)) \hat{g}(p(y)) &=\hat{f}(p(y)) h(p(y)) h(y) \hat{g}(y) \\ \Rightarrow \hat{f}(y)\hat{g}(p(y)) &=\hat{f}(p(y)) \hat{g}(y). 
\end{align}
So we can assume without loss of generality that $f$ and $g$ are co-prime. It follows that the composite polynomials $f \circ p$ and $g \circ p$ are also co-prime. For if $f \circ p$ and $g \circ p$ had the common factor
$l(y)$ it would follow that $f \circ p$ and $g \circ p$ had a common zero since $k$ is algebraically closed. This would imply that $f$ and $g$ had a common zero, contradicting their co-primeness. 

From Equation \ref{eq_4} we see that $f$ must divide $f \circ p$ and $g$ must divide $g \circ p$ . So write $f(p(y)) = e(y) f(y)$ and $g(p(y))= \hat{e}(y) g(y)$. From \ref{eq_4} we see that $e=\hat{e}$. But this implies that 
$e$ is a constant polynomial, since otherwise $f \circ p$ and $g \circ p$ would be co-prime. On the other hand $\deg(f \circ p) = \deg(p) \cdot \deg(f)$, which is a contradiction unless $\deg(f) =0$. The proposition follows.  
\end{proof}

\begin{prop}
 Let $k,\sigma,\delta,a$ be as in Proposition \ref{conj_main}. Then $C_S(a)$ is a free $k[a]$-module of finite rank. 
\end{prop}

\begin{proof}
 This follows directly from Theorem \ref{thm_Centralizer}.
\end{proof}

The following theorem, which as far as the author knows is a new result, follows from what we proved above.  

\begin{theorem}\label{thm_Rich}
 Let $k$ be a field. Let $\sigma$ be an endomorphism of $k[y]$ such that $\sigma(y) =p(y)$, where $\deg(p) >1$, and let $\delta$ be a $\sigma$-derivation. Suppose that $\sigma(\alpha)=\alpha$ and $\delta(\alpha)=0$ for all
$\alpha \in k$. Let $a,b$ be two commuting elements of $k[y][x;\sigma,\delta]$. Then there is a nonzero polynomial $f(s,t) \in k[s,t]$ such that $f(a,b) =0$. 
\end{theorem}

\begin{proof}
 Using the reasoning in the proof of Corollary \ref{cor_Amitsur} this follows from Theorem \ref{thm_Centralizer} and Proposition \ref{conj_main}.
\end{proof}

Note that the center of $k[y][x;\sigma,\delta]$ coincides with $k$ and thus we have a parallel with, for example, Corollary \ref{cor_Amitsur}. We would like to generalize Theorem \ref{thm_Rich} to obtain general conditions under 
which two commuting elements of $S=\OreGen$ are algebraically dependent over the center of $S$. An example of a result in that direction can be found in \cite{HSbook} where the authors prove the following theorem.

\begin{theorem}[\cite{HSbook}, Theorem 7.5]
Let $R=k[y]$, $\sigma(y)=qy$ and $\delta(y)=1$, where $q \in k$ and $q$ is a root of unity.  Form $S=\OreGen$ and let $C$ be the center of $S$. If $a,b$ are commuting elements of $S$ then there is a nonzero polynomial 
$f(s,t) \in C[s,t]$ such that $f(a,b)=0$.  
\end{theorem}

This theorem can not be strengthened to give algebraic dependence over $k$. Indeed, suppose that $q^n=1$. One can check that $x^n$ and $y^n$ commute (in fact they both belong to the center) but they are not algebraically dependent 
over $k$. 

\section*{Acknowledgement}

The author wishes to thank Fredrik Ekstr\"om, Johan \"Oinert and Sergei Silvestrov for helpful comments. 

The author is also grateful to the Dalhgren and Lanner foundations for support. This research was also supported in part by
the Swedish Foundation for International Cooperation in Research and Higher Education (STINT),
The Swedish Research Council (grant no. 2007-6338), The Royal Swedish Academy of Sciences, The Crafoord
Foundation and the NordForsk Research Network ``Operator Algebras and Dynamics'' (grant no. 11580).

\providecommand{\bysame}{\leavevmode\hbox to3em{\hrulefill}\thinspace}
\providecommand{\href}[2]{#2}

\end{document}